\documentclass{amsart}
\usepackage{graphicx}
\usepackage{amscd}
\usepackage{amsmath}
\usepackage{amsfonts}
\usepackage{amssymb}
\theoremstyle{plain}
\newtheorem{theorem}{Theorem}
\newtheorem{corollary}[theorem]{Corollary}
\newtheorem{lemma}[theorem]{Lemma}
\newtheorem{proposition}[theorem]{Proposition}
\theoremstyle{definition}
\newtheorem{example}[theorem]{Example}

\newtheorem{definition}[theorem]{Definition}

\newtheorem{remark}[theorem]{Remark}
\theoremstyle{remark}

\begin{document}
\title{A class of  multiplicative lattices}

\author{Tiberiu Dumitrescu and Mihai Epure}

\address{Facultatea de Matematica si Informatica,University of Bucharest,14 A\-ca\-de\-mi\-ei Str., Bucharest, RO 010014,Romania}
\email{tiberiu\_dumitrescu2003@yahoo.com, tiberiu@fmi.unibuc.ro, }

\address{Simion Stoilow Institute of Mathematics of the Romanian AcademyResearch unit 5, P. O. Box 1-764, RO-014700 Bucharest, Romania}\email{epuremihai@yahoo.com, mihai.epure@imar.ro}


\begin{abstract}
\noindent 
We study the  multiplicative lattices $L$ which satisfy the condition $a=(a:(a:b))(a:b)$ for all $a,b\in L$.
\end{abstract}

\maketitle

\section{Introduction}

 
We recall some standard terminology.
 A {\em multiplicative lattice} is a complete lattice  $(L,\leq)$ (with bottom element $0$ and top element $1$) which is also  a commutative monoid with identity $1$ (the top element) such that
 $$a( \bigvee_\alpha b_\alpha) = \bigvee_\alpha (ab_\alpha)  \mbox{ \ for each } a,b_\alpha\in L.$$
 When $x\leq y$ ($x,y\in L$), we say that $x$ is {\em below} $y$ or that $y$ is {\em above} $x$.
 An element $x$ of $L$ is {\em cancellative} if  $xy = xz$ ($y,z\in L$) implies $y = z$.  For  $x,y\in L$,   $(y : x)$ denotes the element $\bigvee \{a \in L ;\ ax \leq y\}$; so $(y:x)x\leq y$.

 An element $c$ of $L$ is {\em compact} if $c\leq \bigvee S$ with $S\subseteq L$ implies $c\leq \bigvee T$ for some finite subset $T$ of $S$ (here $\bigvee W$  denotes the join  of all elements in   $W$).
 An element  in $L$ is {\em proper} if $x\neq 1$.
 When $1$ is compact, every proper element is below some {\em maximal} element (i.e. maximal in $L-\{1\}$). 
 Let $Max(L)$ denote the set of maximal elements of $L$. By ``$(L,m)$ is  local'', we mean that $Max(L)=\{m\}$.
 A proper element $p$ is {\em prime} if  $xy \leq p$ (with $x,y \in L$) implies $x \leq p$ or $y\leq p$. 
  Every maximal element is prime. $L$ is a {\em (lattice) domain} if $0$ is a prime element.
 An element $x$ is {\em meet-principal} (resp. {\em weak meet-principal})  
 if  $$y \wedge zx = ((y : x) \wedge z) x  \ \ \ \forall  y, z \in L \mbox{ (resp. } (y:x)x=x\wedge y\ \  \forall y\in L).$$
 An element  $x$ is {\em join-principal} (resp. {\em weak join-principal})  if $$y \vee (z : x) = ((yx \vee z) : x) \ \ \forall  y, z \in L \mbox{ (resp. } (xy:x)=y\vee (0:x) \ \ \forall  x\in L).$$
 And $x$ is {\em principal} if it is both meet-principal and join-principal.
 If $x$ and $y$ are principal elements, then so is $xy$. The converse is also true if $L$ is a lattice domain and $xy\neq 0$. 
 In a lattice domain, every nonzero principal element is cancellative.
 The lattice  $L$ is {\em principally generated} if every element is a join of principal elements.
 $L$ is a {\em $C$-lattice} if $1$ is compact, the set of compact elements  is closed under  multiplication  and  every element is a join of compact elements. In a $C$-lattice, every principal element is compact.
 
 The $C$-lattices have a well behaved localization theory.  Let $L$ be a $C$-lattice and $L^*$ the set of its compact elements. For $p \in L$ a prime element and $x\in L$, 
 the {\em localization} of $x$ at $p$ is
 $$x_p = \bigvee \{a \in L^* ;\ as \leq x \mbox{ for some } s \in L^* \mbox{ with } s\not\leq p\}.$$  
 Then $L_p:=\{ x_p;\ x\in L\}$ is  again a lattice with multiplication $(x,y)\mapsto (xy)_p$, join $\{ (b_\alpha)\}\mapsto 
 (\bigvee b_\alpha)_p $ and meet $\{ (b_\alpha) \}\mapsto  (\bigwedge b_\alpha)_p$.
 For $x,y\in L$, we have:

 $\bullet$ $x\leq x_p$, $(x_p)_p=x_p$, 
 $(x \wedge y)_p = x_p \wedge y_p $,
 and $x_p = 1$ iff $x\not\leq p$.
 
 $\bullet$  $x=y$ iff $x_m=y_m$ for each $m\in Max(L)$.
 
 $\bullet$ $(y : x)_p \leq (y_p : x_p)$ with equality if $x$ is compact.
 
 $\bullet$ The set of compact elements of $L_p$ is $\{c_p;\ c\in L^*\}$.
 
 $\bullet$ A compact element $x$ is principal iff $x_m$ is principal for each $m\in Max(L)$ (as usual, ``iff''  stands for ``if and only if'').

  In \cite{ADE} a study of sharp integral domains was done. 
 An  integral domain $D$ is a  {\em  sharp domain} if whenever $A_1A_2\subseteq B$ with $A_1,A_2,B$ ideals of $D$, we have a  factorization $B=B_1B_2$   with $  B_i\supseteq A_i$ ideals of $D$, $i=1,2$.
 In the present paper we extend almost all results in \cite{ADE} to the setup of  multiplicative lattices. Our key definition is the following.
 \begin{definition}
 A lattice $L$ is a  {\em  sharp lattice} if whenever $a_1a_2\leq b$ with $a_1,a_2,b\in L$, we have a  factorization $b=b_1b_2$   with $a_i\leq b_i\in L$, $i=1,2$.
\end{definition}
 In Section 2 we work in the setup of $C$-lattices (simply called  lattices). After obtaining some basic facts (Propositions \ref{e11} and \ref{e21}), we show that if $(L,m)$ is a local sharp lattice  and $m=x_1\vee\cdots \vee x_n$ with  $x_1$,...,$x_n$  join principal elements, then $m=x_i$ for some $i$ (Theorem \ref{e18}).
 While a lattice whose elements are principal is trivially a sharp lattice (Remark \ref{e20}), the converse is true  in a principally generated lattice whose elements are compact (Corollary \ref{e24}).
 
 In Section 3, we work in the setup of $C$-lattice domains generated by principal elements (simply called  lattices). 
 It turns out that every nontrivial totally ordered sharp lattice is isomorphic to the ideal lattice of a valuation domain with value group $\mathbb{Z}$ or $\mathbb{R}$
   (Theorem \ref{e8}).
 A nontrivial sharp lattice $L$ is Pr\"ufer (i.e. its compacts are principal)    of dimension one (Theorem \ref{e15}), thus the  localizations at its maximal elements are   totally ordered sharp  lattices. The converse is true if $L$ has finite character (Definition \ref{e25}), because in this case $(a:b)_m=(a_m:b_m)$ for all $a,b\in L-\{0\}$ and $m\in Max(L)$, see Proposition \ref{e10}. A countable sharp lattice has all its elements principal (Corollary \ref{e26}).
 
 For basic facts or terminology, our  reference papers  are \cite{A} and \cite{OR}.

\section{Basic results}

{\em In this section, the term {\em lattice} means a  $C$-lattice.}

  We begin by giving several characterizations for the sharp lattices. As usual, we say that $a$ {\em divides} $b$ (denoted $a|b$) if $b=ac$ for some $c\in L$.
  
\begin{proposition}\label{e11}
 For a lattice $L$ the following are equivalent:
 
 $(i)$ $L$ is sharp.
 
 $(ii)$ $a=(a:(a:b))(a:b)$ for all $a,b\in L$.
 
 $(iii)$ $(a:b)|a$ for all $a,b\in L$.
 
 $(iv)$ $(a:b)|a$ whenever $a,b\in L$,   $0<a<b<1$   and $a$ is not prime.
\end{proposition}
\begin{proof}
$(i) \Rightarrow (iii)$.  
 Since $(a:(a:b))(a:b)\leq a$, and $L$ is sharp, we have  a factorization $a=a_1a_2$ with $a_1\geq (a:(a:b))$ and $a_2\geq (a:b)$.
 We get 
 $$a_2\leq (a:a_1)\leq (a:(a:(a:b)))=(a:b)\leq a_2$$ 
 where the  equality is easy to check, so $(a:b)=a_2$  divides $a$.

   $(iii) \Rightarrow (ii)$. From $a=x(a:b)$ with $x\in L$, we get 
   $x\leq (a:(a:b))$, so $$a=x(a:b)\leq (a:(a:b))(a:b)\leq a.$$
   
$(ii) \Rightarrow (i)$. 
 Let $a_1,a_2,b\in L$ with $b\geq a_1a_2$. By $(ii)$ we get 
 $b=(b:(b:a_1))(b:a_1)$ and clearly $a_1\leq (b:(b:a_1))$ and $a_2\leq (b:a_1)$.
 
 $(iv) \Leftrightarrow (iii)$ follows from observing that: $(1)$  $(a:b)=(a:(a\vee b))$ and  $(2)$ $(a:b)\in\{a,1\}$ if $a$ is prime. 
 
\end{proof}

\begin{proposition} \label{e21}
 If $L$ is a sharp lattice and $m\in L$ a maximal element, 
 there is no element properly between $m$ and $m^2$.
\end{proposition}
\begin{proof}
 If $m^2 < x < m$, then $(x:m)=m$, so $(x:(x:m))=m$, thus $x=(x:m)(x:(x:m))=m^2$, a contradiction, cf.  Proposition \ref{e11}. 
\end{proof}

Recall that a ring $R$ is a {\em special primary ring} if $R$ has a
unique maximal ideal $M$ and if each proper ideal of $R$ is a power of $M$, see \cite[page 206]{LM}.

\begin{corollary}
 The ideal lattice of a  Noetherian commutative unitary ring $R$ is sharp iff $R$ is a finite direct product of Dedekind domains and special primary rings. 
\end{corollary}
\begin{proof}
 Combine Propositions \ref{e11} and \ref{e21} and \cite[Theorem 39.2, Proposition 39.4]{G}.
\end{proof}

\begin{remark}\label{e20}
Let $L$ be a lattice.

 $(i)$ If all elements of $L$ are weak meet principal, then $L$  is sharp (Proposition \ref{e11}). In particular, this happens when
 $a\wedge b=ab$ for all $a,b\in L$.
 

 $(ii)$ If $L$ is  sharp, then every $p\in L-\{1\}$  whose only    divisors are $p$ and $1$  is a prime element, because  $(p:b)=p$ or $1$ for all $b\in L$ (Proposition   \ref{e11}).
 The converse is not true. Indeed, let $L$ be the lattice  $0<a<b<c<1$ with $a^2=b^2=ab=0$, $ac=a$, $bc=b$, $c^2=c$. Here every $x\in L-\{c,1\}$ has nontrivial factors, while the lattice is not sharp because $(a:b)=b$ does not divide $a$. 
 
 
 $(iii)$ A finite lattice $0< a_1 < \cdots < a_n < 1$, $n\geq 2$, is sharp provided $a_{i+1}^2\geq a_i$ for $1\leq i\leq n-1$.  By Proposition \ref{e11}$(iv)$, it suffices to show that  whenever $(a_i:a_j)=a_k$ with $1\leq i< j,k\leq n$, it follows that $a_k$ divides $a_i$. Indeed, from $(a_i:a_j)=a_k$, we get $a_ja_k\leq a_i\leq a_{i+1}^2\leq a_ja_k$, so $a_i=a_ja_k$.
 
 $(iv)$ Using similar arguments, it can be shown that a   lattice whose poset is $0< a<b<c < 1$ is sharp iff $c^2\geq b$ and either $b^2\geq a$ or ($b^2=0$ and $bc=a$). In this case, a  computer search finds $13$ sharp lattices out of $22$ lattices.
\end{remark}

We give the main result of this section.

\begin{theorem}\label{e18}
 Let $L$ be a sharp lattice.
 
 $(i)$ If  $x,y\in L$ are join principal elements and  $(x:y)\vee (y:x) = x\vee y$, then $x\vee y=1$.
 
 $(ii)$ If  $(L,m)$ is local and $m=x_1\vee\cdots \vee x_n$ with  $x_1$,...,$x_n$  join principal elements, then $m=x_i$ for some $i$.
\end{theorem}
\begin{proof} 
 $(i)$ Since $L$ is sharp and $(x\vee y)^2 \leq x^2\vee y$, we can factorize $x^2\vee y=ab$     with $x\vee y \leq a \wedge b$.
 Since $x$ is join principal and  $(y:x)\leq x\vee y$, we get 
 $$x\vee y\leq a\leq (x^2\vee y):b\leq (x^2\vee y):(x\vee y)=
 (x^2\vee y):x=x\vee (y:x)=x\vee y.$$    Thus $a=x\vee y=b$, as $a$ and $b$ play symmetric roles.  So $x^2\vee y=ab=(x\vee y)^2$. As $y$ is join principal and  $(x^2:y)\leq (x:y)\leq x\vee y$, we finally get
 $$1=((x^2\vee xy \vee y^2):y)=(x^2:y)\vee x\vee y =x\vee y.$$ 
 
 $(ii)$ Suppose that $n\geq 2$ and no $x_i$ can be deleted from the given representation $m=x_1\vee\cdots \vee x_n$. 
 It is straightforward to show that a factor lattice of a sharp lattice is again sharp. 
 Modding out by $x_3\vee \cdots \vee x_n$, we may assume that $n=2$. As $(x_1:x_2)\vee (x_2:x_1)\leq m=x_1\vee x_2$, we get a contradiction from $(i)$. 
\end{proof}

Before giving an application  of Theorem \ref{e18}, we insert a simple lemma.

\begin{lemma} \label{e7}
 Let $L$ be a sharp lattice and $p\in L$ a prime element. If $L$ is sharp, then so is  $L_p$. 
\end{lemma}
\begin{proof}
 Let $a_1,a_2,b\in L$ with $(a_1a_2)_p\leq b_p$. As $L$ is sharp, we have   $b_p=c_1c_2$ for some $a_i\leq c_i\in L$ ($i=1,2$), so $b_p=(c_1c_2)_p$ and $(a_i)_p\leq (c_i)_p$.
\end{proof}
 
 Following \cite{AJ}, we say that a lattice $L$ is {\em weak Noetherian} if it is  principally generated and each $x\in L$ is compact. 

\begin{corollary}\label{e24}
 Let $L$ be a   weak Noetherian lattice. Then $L$ is sharp iff  its elements are   principal.
\end{corollary}
\begin{proof}
 The ``only if part'' is covered by Remark \ref{e20}$(i)$. For the converse, pick an arbitrary  maximal element $m\in L$.  It suffices to prove thet $m$ is principal \cite[Theorem 1.1]{AJ}. As $m$ is compact, we can check this property locally  \cite[Lemma 1.1]{AJ}), so we may assume that $L$ is local (Lemma \ref{e7}). Apply Theorem \ref{e18}$(ii)$ to complete.
\end{proof}

\section{Sharp lattice domains}


 {\em In this section, the term {\em lattice} means a  $C$-lattice domain generated by principal elements.}

 First we introduce an ad-hoc definition.

\begin{definition}\label{e22}
 A  lattice   $L$ is a {\em pseudo-Dedekind lattice} if $(x:a)$ is a principal element whenever  $x,a\in L$ and $x$ is   principal.
\end{definition}

\begin{proposition} \label{e5}
Every sharp lattice is  pseudo-Dedekind. 
\end{proposition}
\begin{proof}
 The assertion follows from Proposition \ref{e11}, because a factor of a principal element is principal \cite[Lemma 3.2]{AJ}. 
\end{proof}

\begin{example}
 There exist pseudo-Dedekind lattices which are not sharp. For instance, let $M$ be the (distributive)   lattice   of all ideals of the multiplicative monoid $\mathbb{N}_0=\mathbb{N}\cup\{0\}$ \cite[page 138]{A}. Every $a\in M$ has the form $a=\bigcup \{ y\mathbb{N}_0|\ y\in S\}$, for some $S\subseteq \mathbb{N}_0$.
 If $x\in \mathbb{N}_0$, then $(x\mathbb{N}_0:a)=\bigcap \{(a\mathbb{N}_0:y)|\ s\in S\}=z\mathbb{N}_0$ (for some $z\in \mathbb{N}_0$) is a principal element. So  $M$ is a pseudo-Dedekind lattice. But $M$ is not sharp because  for  $a=4\mathbb{N}_0\cup 9\mathbb{N}_0$ and $b=2\mathbb{N}_0\cup 3\mathbb{N}_0$, we get $(a:b)=b^2$ and $(a:(a:b))=b$, so $(a:b)(a:(a:b))=b^3\neq a.$
 See also \cite[Example 8]{ADE} for a ring-theoretic example of this kind.
\end{example}

 A lattice  $L$ is a {\em  Pr\"ufer lattice} if every compact element of $L$ is principal. It is well known \cite[Theorem 3.4]{A} that $L$ is a    Pr\"ufer lattice iff  $L_m$ is totally ordered for each maximal element  $m$.
We show that a sharp lattice is Pr\"ufer. 
 
 \begin{remark}\label{e19} 
 If $L$ is a pseudo-Dedekind lattice, then  the set $P$  of all principal elements of $L$  is a cancellative GCD monoid in the sense of \cite[Section 10.2]{H}. Indeed, the LCM of two elements $x,y\in P$ is $x\wedge y=y(x:y)$.
\end{remark}
 
 \begin{proposition} \label{e28}
Every sharp lattice is  Pr\"ufer. 
\end{proposition}
\begin{proof}
 As  $L$ is principally generated, it suffices to show that $a\vee b$ is a principal element for each pair  of nonzero principal elements $a,b\in L$.
 Dividing $a,b$ by their GCD (Remark \ref{e19}),  we may assume that $(a:b)=a$ and $(b:a)=b$. Then $a\vee b=1$ (Theorem \ref{e18}). 
\end{proof}

\begin{example}\label{e27}
 Let $\mathbb{Z}_-$ denote the set of all integers $\leq 0$ together with the symbol $-\infty$.
 Then $\mathbb{Z}_-$ is a lattice under the usual addition and order. Note that  $\mathbb{Z}_-$ is isomorphic to the ideal lattice of a discrete valuation domain, so $\mathbb{Z}_-$ is sharp.
 
 Let $\mathbb{R}_1$ denote the set of all intervals $(r,\infty]$ and $[r,\infty]$ for $r\in \mathbb{R}$ together with  $\{\infty\}$.
 Then $\mathbb{R}_1$ is a lattice under the usual interval addition and inclusion. 
 To show that $\mathbb{R}_1$ is sharp, it suffices to check that $a=(a:(a:b))(a:b)$ for all $a,b\in \mathbb{R}_1-\{\{\infty\}\}$ with  $a\leq b$, cf. Proposition \ref{e11}.  This is done in the table below.
 
 \begin{tabular}{cccccccc}
  $a$ & \vline & $b$ & \vline & $a:b$ & \vline & $(a:(a:b))$ \\
  \hline
  \newline
  
  $[r,\infty]$ & \vline & $[t,\infty]$ & \vline & $[r-t,\infty]$ & \vline & $[t,\infty]$ \\
  
  $(r,\infty]$ & \vline & $(t,\infty]$ & \vline & $[r-t,\infty]$ & \vline & $(t,\infty]$ \\
  
  $[r,\infty]$ & \vline & $(t,\infty]$ & \vline & $[r-t,\infty]$ & \vline & $[t,\infty]$ \\
  
  $(r,\infty]$ & \vline & $[t,\infty]$ & \vline & $(r-t,\infty]$ & \vline & $(t,\infty]$. \\
  
 \end{tabular}
 \\
 Note that $\mathbb{R}_1$ is isomorphic to the ideal lattice of a valuation domain with value group $\mathbb{R}$.
\end{example}

We embark to show that every nontrivial totally ordered sharp lattice is isomorphic to $\mathbb{Z}_-$ or $\mathbb{R}_1$ above.
Although the  following lemma  is   known, we insert a proof for reader's convenience.

\begin{lemma}\label{e3}
 Let $(L,m)\neq \{0,1\}$ be a totally ordered lattice  and   $p\in L$, $0\neq p\neq m$, a prime element. Then
 
 $(i)$  $p$ is not principal.
 
 $(ii)$ $(z:(z:p))=p$ for each nonzero principal element $z\leq p$.
 
 $(iii)$ If  $L$ is also  pseudo-Dedekind, then $Spec(L)=\{0,m\}$.
\end{lemma}
\begin{proof}
 As $p\neq m$, there exists a principal element $p<y\leq m$.

$(i)$  As $y$ is principal, we get  $p=y(p:y)=yp$,  because    $p$ is prime so $p=(p:y)$. Hence $p$ is not cancellative, so it is not principal.
 
$(ii)$ Let  $z\leq p$ be a nonzero principal element. 
Note that $(z:(z:p))\neq 1$, otherwise  $zy=(z:p)y\geq (z:y)y=z$, so $zy=z$, a contradiction because $z$ is cancellative.
Since  $p\leq (z:(z:p))$, it suffices to show that $x\not\leq (z:(z:p))$ for each principal $x\not\leq p$. As $p$ is prime, we have $z\leq p<x^2$.
If $x\leq (z:(z:p))$, then $x(z:p)\leq z$, so $z=x^2(z:x^2)\leq x^2(z:p)\leq z$, hence $z\leq zx$, thus $x=1$, a contradiction.
\end{proof}

\begin{theorem}\label{e8}
 For a totally ordered lattice $L\neq \{0,1\}$, the following  are equivalent:
 
 $(i)$ $L$ is sharp.
 
 $(ii)$ $L$ is pseudo-Dedekind.
 
 $(iii)$ $L$ is isomorphic to $\mathbb{Z}_-$ or $\mathbb{R}_1$ of Example \ref{e27}.   
\end{theorem}
\begin{proof}
 $(i) \Rightarrow (ii)$ follows from Proposition \ref{e5}.
 $(ii) \Rightarrow (iii)$
 Let $m$ be  the maximal element of $L$.  Let $G$ be the monoid of nonzero  principal elements of $L$.
 Then $G$ is a cancellative totally ordered monoid  with respect to
 the opposite of order induced from $L$. 
 Let $a,b\in G$. Since $L$ is totally ordered, we get that $a$ divides $b$ or $b$ divides $a$.
 Moreover, since $Spec(L)=\{0,m\}$ (Lemma \ref{e3}), $a$ divides some power of $b$. By \cite[Proposition 2.1.1]{EP}, the quotient group of $G$ can be embedded as an ordered subgroup  $K$ of $(\mathbb{R},+)$; hence $K$ is cyclic or dense in $\mathbb{R}$.  If $K$ is cyclic it follows easily that $L$ is isomorphic to $\mathbb{Z}_-$  of Example \ref{e27}.   
 Suppose that $K$ is dense in $\mathbb{R}$, so there exists an ordered monoid embedding $v:G\rightarrow \mathbb{R}_{\geq 0}$ with dense image. 
 We claim that $v$ is onto. Deny, so   there exists a positive real  $g\notin v(G)$. Let $a\in G$ with $v(a) > g$ and set  
 $b:=\bigvee \{ x\in G\ |$ $v(x) \geq g \}$. Since $L$ is pseudo-Dedekind, it follows that $c=(a:b)$ is a principal element. On the other hand, a straightforward computation shows that 
 \begin{equation}\label{e23}
  c=\bigvee \{ y\in G\cap L\ |\ v(x) \geq  v(a)-g\}
 \end{equation}
  so $v(c) \geq v(a)-g$, in fact $v(c)>v(a)-g$, because $g\notin v(G)$.  As $G$ is dense in $\mathbb{R}$, there exists some $d\in G\cap L$ with $v(c)>v(d)>v(a)-g$, so $c<d$. On the other hand formula (\ref{e23}) gives $d\leq c$, a contradiction. It remains that $v(G)=\mathbb{R}_{\geq 0}$.
 Now it is easy to see that sending $[r,\infty]$ into $v^{-1}(r)$ and $(r,\infty]$ into $\bigvee \{ x\in G\ |$ $v(x) \geq r \}$ we get a lattice isomorphism from $\mathbb{R}_1$ to $L$.
 The implication  $(iii) \Rightarrow (i)$ follows from  Example \ref{e27}.   

\end{proof}
 We prove the main result of this paper.

\begin{theorem}\label{e15}
 Let  $L\neq \{0,1\}$ be    a sharp lattice.  
  Then $L_m$   is isomorphic to $\mathbb{Z}_-$ or $\mathbb{R}_1$ (see Example \ref{e27}) for every  $m\in Max(L)$ and $L$ is a  one-dimensional Pr\"ufer lattice.
\end{theorem} 
\begin{proof}
 As $L$ is a  Pr\"ufer lattice (Proposition \ref{e28}), we may change $L$ by $L_m$ and thus assume that $L$ is totally ordered and sharp (Lemma \ref{e7}). Apply  Theorem \ref{e8} and Lemma \ref{e3} to complete.
\end{proof}



We extend the concepts of ``finite character'' and ``h-local'' from integral domains to lattices.

\begin{definition}\label{e25}
 Let $L$ be a lattice.
 
 (i) $L$ has {\em finite character} if  every  nonzero  element is below only finitely many maximal elements. 
 
 (ii) $L$ is   {\em h-local} if it has finite character and  every nonzero prime element is below a unique maximal element.
\end{definition}

 
 The next result extends \cite[Lemma 3.5]{O} to lattices.
 
\begin{proposition}\label{e10}
 Let $L$ be an h-local lattice. If $a,b\in L-\{0\}$ and $m\in Max(L)$, then $(a:b)_m=(a_m:b_m)$.
\end{proposition}
\begin{proof}
 We first prove two claims.
 
 {\em Claim $1$.} If $n\in Max(L)-\{m\}$, then $a_n\not\leq m$.
 \\
 Suppose that $a_n\leq m$.  
 Let $S$ be the set of all products $bc$ where $b,c\in L$ are  compact elements with $b\not\leq m$ and $c\not\leq n$. Note that $S$ is multiplicatively closed. Moreover $a$ is not above any member of $S$. 
 Indeed, if  $bc\leq a$ and  $c\not\leq n$, then $b\leq a_n\leq m$. 
 By \cite[Theorem 2.2]{A} and its proof, there exits a prime element $p\geq a$ such that $p$ is not above any member of $S$.  It follows that $p\leq m\wedge n$, which is a contradiction, because $L$ is h-local. Indeed, if $p\not\leq m$, then $b\not\leq m$ for some compact  $b\leq p$, so $b=b\cdot 1\in S$
 Thus Claim $1$ is proved.
 
 {\em Claim $2$.} The element $s:=\bigwedge\{ a_n| n\in Max(L),n\neq m\}$ is not below $m$.
 \\
 Indeed, as $L$ is h-local, $a$ is below only finitely many maximal elements $n_1$,...,$n_k$ distinct from $m$, hence $s= a_{n_1}\wedge\cdots \wedge a_{n_k}$
  By Claim $1$, $s$ is not below $m$,  thus proving Claim $2$.  To complete the proof, we use  element $s$ in Claim 2 as follows. We have 
  $$sb(a_m:b_m)\leq \bigwedge\{ a_q| q\in Max(L)\}=a$$ so $s(a_m:b_m)\leq (a:b)$, hence $(a_m:b_m)\leq (a:b)_m$, because $s\not\leq m$. Since clearly $(a:b)_m\leq (a_m:b_m)$, we get the result.
\end{proof}

\begin{theorem}
 For a  finite character lattice $L\neq \{0,1\}$, the following are equivalent: 
 
 $(i)$ $L$ is  sharp.

 $(ii)$ $L_m$   is isomorphic to $\mathbb{Z}_-$ or $\mathbb{R}_1$ (see Example \ref{e27}) for every  $m\in Max(L)$.
\end{theorem}
\begin{proof}
 $(i)$ implies $(ii)$ is covered by Theorem \ref{e15}.
 
 $(ii)$ implies $(i)$.  From $(ii)$ we derive that $L$ has Krull dimension one, so $L$ is h-local. Let $a,b\in L-\{0\}$. It suffices to check locally the equality $a=(a:(a:b))(a:b)$. But this follows from  Theorem  \ref{e8} and Proposition  \ref{e10}.
\end{proof}

%

Say that elements $a,b$ of a lattice $L$ are {\em comaximal} if $a\vee b=1$.
The following result is \cite[Lemma 4]{D}. 

\begin{lemma}\label{e14}
 Let $L$ be a  lattice and $z\in L$ a  compact element which is below infinitely many maximal elements.
 There exist an infinite set $\{ a_n;\ n\geq 1\}$ of pairwise comaximal proper compact elements such that $z\leq a_n$ for each $n$.
\end{lemma}

\begin{proposition}\label{e12} 
Any  countable pseudo-Dedekind Pr\"ufer lattice  $L$ has  finite character.
\end{proposition}
\begin{proof}
 Suppose on the contrary  there exists a nonzero  element $z\in L$ which is below  infinitely many maximal elements. Since $L$ is principally generated, we may assume that $z$ is principal.  By Lemma \ref{e14}, there exists  an infinite set $(a_n)_{n\geq 1}$ of proper pairwise comaximal compact elements above $z$. As $L$ is Pr\"ufer, each $a_n$ is principal.
 Since $L$ countable, we get $\tau:=\bigwedge_{n\in A} a_n = \bigwedge_{n\in B} a_n$ for two  nonempty subsets $B\not\subseteq A$ of $\mathbb{N}$. 
 Pick  $k\in B-A$, so $a_k\geq \tau$. Since every $a_n$ is above $z$, we get $z=a_nb_n$ for some nonzero principal element $b_n\in L$ and $(z:b_n)=a_n$. We have 
 $$\tau = \bigwedge_{n\in A} a_n = \bigwedge_{n\in A} (z:b_n) = (z:\bigvee_{n\in A} b_n)$$ so $\tau$
 is a principal element because $L$ is pseudo-Dedekind. 
   From  $a_k\geq \tau$, we get $\tau =a_kc$  for some nonzero principal  element $c\in L$.   Hence
  $$c\leq (\tau : a_k) =            
  \bigwedge_{n\in A} (a_n : a_k) = 
  \bigwedge_{n\in A} a_n  =  \tau = a_kc
  $$
  because $a_n\vee a_k=1$ for each $n\in A$.    From  $a_kc=c$, we get $a_k=1$, which is a contradiction. 
 
\end{proof}

  A lattice  $L$ is a {\em Dedekind lattice} if every  element of $L$ is principal.

\begin{corollary}\label{e26}
 A countable sharp lattice  $L$ is a Dedekind lattice.
\end{corollary}
\begin{proof}
  Let $m\in Max(L)$. As $L_m$ is countable, Theorem \ref{e15} implies that $L_m$ is isomorphic to $\mathbb{Z}_-$, so  each element of  $L_m$ is principal.
  By Proposition \ref{e12}, $L$ has  finite character. It follows easily that every element of $L$ is compact and  locally principal, hence principal.

\end{proof}

Our concluding remark is in the spirit of \cite[Remark 4.7]{OR}.

\begin{remark}
 Let $L$ be a Pr\"ufer lattice. Then $L$ is modular because it is locally totally ordered. By \cite[Theorem 3.4]{A}, $L$ is isomorphic to the lattice of ideals of some Pr\"ufer integral domain. In particular, it follows that a sharp lattice  is isomorphic to the lattice of ideals of some sharp integral domain.
\end{remark}



\begin{thebibliography}{11111}

 \bibitem{ADE} Z. Ahmad, T. Dumitrescu and M. Epure, A Schreier  domain type  condition,
 Bull.  Math.  Soc.  Sci.  Math.  Roumanie, 55(103) (2012), 241-247.
 
 \bibitem{A} D.D. Anderson, Abstract commutative ideal theory without chain conditions, Algebra Univ. 6 (1976), 131-145.
 
 \bibitem{AJ} D.D. Anderson and C. Jayaram, Principal element lattices,
 Czech. Math. J. 46 (1996), 99-109.

 \bibitem{D} T. Dumitrescu,  A Bazzoni-type theorem for multiplicative lattices, will appear in Graz 2018 Conference Proceedings. 
 
 \bibitem{EP} A. Engler and A. Prestel, {\em Valued Fields}, Springer, Berlin Heidelberg 2005
 
 \bibitem{G} R. Gilmer, {\em Multiplicative Ideal Theory}, Marcel Dekker, New York, 1972.

 \bibitem{H} F. Halter-Koch, {\em Ideal Systems: an Introduction to Multiplicative Ideal Theory}, Marcel Dekker, New York, 1998.
 
 \bibitem{LM} M. Larsen and P.  McCarthy, {\em Multiplicative Theory of Ideals}, Academic Press, New York, 1971.
 
 \bibitem{O} B. Olberding,
Globalizing local properties of Pr\"ufer domains, J. Algebra  205 (1998), 480-504.
 
 \bibitem{OR}  B. Olberding and A. Reinhart, Radical factorization in commutative rings, mo\-no\-ids and multiplicative lattices,  Algebra Univers. (2019), 80:24.
 
\end{thebibliography}
\end{document}